\documentclass[10pt]{amsart} 
\usepackage{verbatim}
\usepackage{amssymb}
\usepackage{enumerate}
\usepackage{tikz}
\usepackage[utf8]{inputenc}
\numberwithin{equation}{section}

\usepackage{geometry}
\newtheorem{theorem}{Theorem}[section]

\newtheorem{lemma}[theorem]{Lemma}
\newtheorem{fact}[theorem]{Fact}
\newtheorem{dfn}[theorem]{Definition}

\newtheorem{obs}[theorem]{Observation}

\newtheorem{claim}{Claim}[theorem]

\theoremstyle{definition}

\theoremstyle{remark}

\def\myheads#1;#2;{
\pagestyle{myheadings}
\markboth{{\sc\hfill #1\hfill\protect\makebox[0cm][r]{\rm\today}}}
{{\sc\protect\makebox[0cm][l]{\rm\today}\hfill #2\hfill}}
}

\newif\ifdeveloping

\developingtrue%Header with date

\ifdeveloping
\fi

\newif\ifcommented
\newcommand{\comm}[1]{}

\ifcommented
\renewcommand{\comm}[1]{
\fbox{\fbox{\begin{minipage}{300pt}#1\end{minipage}}
}}

\fi

\newcommand{\mc}[1]{\mathcal{#1}}
\newcommand{\mbb}[1]{\mathbb{#1}}

\newcommand{\uhp}{\upharpoonright}

\newcommand{\omg}{{\omega_1}}

\newcommand{\gr}{G=(V,E)}
\newcommand{\setm}{\setminus}

\newcommand{\subs}{\subset}

\def\<{\left\langle}
\def\>{\right\rangle}
\def\br#1;#2;{\bigl[ {#1} \bigr]^ {#2} }

\newcommand{\oo}{{\omega}}

\author[D. T. Soukup]{D\'aniel T. Soukup}

\address
      {University of Toronto, Toronto, ON, Canada}
\email{daniel.soukup@mail.utoronto.ca}
\urladdr{http://www.math.toronto.edu/$\sim$dsoukup}

\title{Davies-trees in infinite combinatorics}
\date{\today}

\begin{document}

\maketitle

\begin{abstract}
This short note, prepared for the Logic Colloquium 2014, provides an introduction to Davies-trees and presents new applications in infinite combinatorics. In particular, we prove that: 
every $n$-almost disjoint family of sets is essentially disjoint for any  $n\in \mbb N$ \cite{kopefam}; $\mbb R^2$ is the union of $n+2$ clouds if the continuum is at most $\aleph_n$ for any $n\in \mbb N$ \cite{clouds}; every uncountably chromatic graph contains $n$-connected uncountably chromatic subgraphs for every $n\in \mbb N$ \cite{kopeconn}.
\end{abstract}

\section{Introduction}

The goal of this paper is to introduce the reader to a somewhat overlooked set theoretic tool involving sequences 
(but not chains) of elementary submodels and their applications. The literature contains several well written introductions to (chains of) elementary submodels and their applications in topology and combinatorics; see the papers \cite{dowel, gesel, soukel}, the book \cite{justel}, the presentation \cite{sigmaD} or the blog post \cite{mike}. Nowadays every other proof in set theory and general topology uses elementary submodels and we will hence assume basic familiarity with this tool. Nonetheless, we include a very short, and over-simplified, introduction: we will work with elementary submodels $M$ of $H(\Theta)$ (sets of hereditary cardinality $<\Theta$ for some large enough cardinal $\Theta$). $H(\Theta)$ captures a large fragment of the set theoretic universe (i.e. almost all of ZFC is satisfied) and $M$ being an elementary submodel means that a formula $\phi$ with parameters from $M$ is true in $M$ iff it is true in $H(\Theta)$. 

How are elementary submodels useful? If a structure $\mc X$ of arbitrary size is intersected with say a countable elementary submodel $M$ so that $\mc X\in M$ then the structure $\mc X\cap M$ will be very similar to $\mc X$  but has countable size; we say that properties of $\mc X$ \emph{reflect} to $M\cap \mc X$. It is easy to imagine that such a construction is useful in many situations.   

Why are there always elementary submodels which have all the parameters we need for a certain proof? The downward L\"owenheim-Skolem theorem says that whenever $\mc A\subs H
(\Theta)$ is countable then we can find a countable elementary submodel $M$ of $H(\Theta)$ so that $\mc A\subs M$. i.e. $M$ contains everything relevant to our particular situation. We regularly use the following

\begin{fact}\label{elfact} Suppose that $M$ is an elementary submodel of $H(\Theta)$ and $X\in M$. If $X$ is countable then $X\subseteq M$ or equivalently, if $X\setm M$ is nonempty then $X$ is uncountable.
\end{fact}

In the next sections, we continue with defining Davies-trees and explaining the simplest way of constructing them. Next, we review applications of Davies-trees prior our work. We continue with three new applications in the form of alternate and simplified proofs to well known theorems from infinite combinatorics. We finish with a couple of remarks on work in progress and further applications. 

We believe that the use of Davies-trees, beyond their simplicity, provides the \emph{right} way of explaining and proving several results (including the ones presented below); we encourage the interested reader to compare our work with the original proofs of these results.

\section{Chains versus trees of elementary submodels}

Roughly speaking, single elementary submodels are generally used for reflection arguments while chains of elementary submodels usually provide the base for an inductive or recursive proof. In other words, we express our structure as an increasing chain of nice substructures (i.e. small in cardinality but similar in nature) and the transfinite induction/recursion follows an enumeration which is based on this chain. In cases when we are forced to use increasing chains of \emph{countable} elementary submodels our tool becomes restricted to structures of size $\omega_1$. Hence, in many cases, the Continuum Hypothesis\footnote{``independence reared its ugly head'' as P. Erd\H os used to say.} appears if the structure at hand has size $2^\oo$.

Is there a natural way to cover large uncountable structures by countable elementary submodels in a useful way i.e. in a way so that our tricks and tools from previous applications transfer? This was done by Roy. O. Davies \cite{davies}, and hence the name \emph{Davies-trees}, while answering a question of Sierpinski; see Section \ref{first} for further details.

The simple idea is that we can always cover a structure $\mc X$ with a continuous chain of elementary submodels of size $<|\mc X|$ so lets see what happens if we repeat this process and cover each elementary submodel again with chains of smaller submodels, and those submodels with chains of smaller submodels and so on...

\begin{fact}\label{main} Suppose that $\mc A$ is a countable set and $\mc X$ is an arbitrary set. Then there is a large enough cardinal $\Theta$ and a sequence of $\mc M=(M_\alpha)_{\alpha<\kappa}$ of countable elementary submodels of $H(\Theta)$ so that
\begin{enumerate}
	\item $\{\mc X\}\cup \mc A\subs M_\alpha$ for all $\alpha<\kappa$,
	\item ${\mc X\subs \bigcup_{\alpha<\kappa} M_\alpha}$,
	\item for every $\beta<\kappa$ there is $m_\beta\in \mbb N$ and models $N_{\beta,i}\prec H(\Theta)$ such that $ \{\mc X\}\cup \mc A\subs N_{\beta,i}$ for $i<m_\beta$ and $$\bigcup\{M_\alpha:\alpha<\beta\}=\bigcup\{N_{\beta,i}:i<m_\beta\}.$$
\end{enumerate}
\end{fact}

We will refer to a sequence of models $\mc M$ with property (3) as a \emph{Davies-tree}. 

Note that if the sequence $(M_\alpha)_{\alpha<\kappa}$ is increasing then $\bigcup\{M_\alpha:\alpha<\beta\}$ is also an elementary submodel of $H(\Theta)$ for each $\beta<\kappa$; however, there is no way to cover a set of size bigger than $\omg$ with an increasing chain of \emph{countable} sets. Fact \ref{main} says that we can cover by countable elementary submodels and almost maintain the property that the initial segments $\bigcup\{M_\alpha:\alpha<\beta\}$ are submodels. Indeed, each initial segment is the union of \emph{finitely many} submodels by condition (3) while these models contain everything relevant (denoted by $\mc A$ above) as well.

\begin{proof}
Suppose that $\mc X$ has size $\lambda$. We recursively construct a tree $T$ of finite sequences of ordinals and elementary submodels $M(a)$ for $a\in T$. Let $\emptyset\in T$ and let $M(\emptyset)$ be an elementary submodel of size $\lambda$ so that  
\begin{itemize}
	\item $\{\mc X\}\cup \mc A\subs M(\emptyset)$,
	\item ${\mc X\subs M(\emptyset)}$.
\end{itemize}
Suppose that we defined a tree $T'$ and corresponding models $M(a)$ for $a\in T'$. Fix $a\in T'$ and suppose that $M(a)$ is \emph{uncountable}. Find a continuous sequence of elementary submodels $(M(a^\frown\xi))_{\xi<\zeta}$ so that
\begin{itemize}
	\item $\{\mc X\}\cup \mc A\subs M(a^\frown\xi)$ for all $\xi<\zeta$,
	\item $M(a^\frown\xi)$ has size less than $M(a)$.
\end{itemize}
 We extend $T'$ with $\{a^\frown \xi:\xi<\zeta\}$ and iterate this procedure to get $T$.

%\begin{comment}

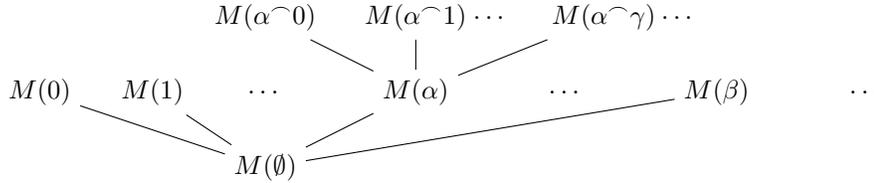
\vspace{0.4 cm}
\hspace{1 cm}
\begin{tikzpicture}[]
% \draw[step=1,very thin,color=gray]  (-5,0) grid (5,5);

\draw (0,0)   node  (me)  {$M(\emptyset)$};

\draw (-3,1)   node   (m0) {$M(0)$};
\draw (-1.5,1)   node (m1)   {$M(1)$};
\draw (0,1)   node    {$\dots$};
\draw (2,1)   node (ma)   {$M(\alpha)$};
\draw (4,1)   node    {$\dots$};
\draw (6,1)   node (mb)   {$M(\beta)$};
\draw (8,1)   node    {$\dots$};

\draw  (me) -- (m0);
\draw  (me) -- (m1);
\draw  (me) -- (ma);
\draw  (me) -- (mb);

\draw (0,2)   node   (ma0) {$M({\alpha}^\frown 0)$};
\draw (2,2)   node (ma1)   {$M({\alpha}^\frown 1)$};
\draw (3,2)   node    {$\dots$};
\draw (4.5,2)   node (maa)   {$M({\alpha}^\frown\gamma)$};
\draw (5.5,2)   node    {$\dots$};
% \draw (6,2)   node (mab)   {$M_{\alpha,{\delta}}$};
% \draw (7,2)   node    {$\dots$};

\draw  (ma) -- (ma0);
\draw  (ma) -- (ma1);
\draw  (ma) -- (maa);
%\draw  (ma) -- (mab);

\end{tikzpicture}

%\end{comment}

\vspace{0.2 cm}

It is easy to see that this process produces a downwards closed subtree $T$ of $Ord^{<\oo}$ and if $a\in T$ is a branch then $M(a)$ is countable. Let us well order $\{M(a): a\in T$ is a branch$\}$ by the lexicographical ordering. 

We wish to show that if $b\in T$ is a branch then $\bigcup\{M(a):a<_{lex}b, a\in T$ is a branch$\}$ is the union of finitely many submodels containing $\{\mc X\}\cup \mc A$. Suppose that $|b|=n\in \mbb N$ and write $$N_{b,i}=\bigcup\{M((b\uhp {i-1})^\frown \xi):\xi<b(i-1)\}$$ for $i=1\dots n$. It is clear that $N_{b,i}$ is an elementary submodel as a union of an increasing chain. Also, if $a<_{lex}b$ then $M(a)\subs N_{b,i}$ must hold where $i=\min\{j\leq n: a(j)\neq b(j)\}$.

\end{proof}

Note that this proof shows that if $\mc X$ has size $\aleph_n$ then every initial segment in the lexicographical ordering is the union of $n$ elementary submodels (the tree $T$ has height $n$).

In the future, when working with a sequence of elementary submodels $\mc M=(M_\alpha)_{\alpha<\kappa}$, we use the notation $$\mc M_{<\beta}=\bigcup\{M_\alpha:\alpha<\beta\}$$ for $\beta<\kappa$.

\section{The first applications}\label{first}

\subsection{The very first} As we mentioned already, the above constructed tree of models is originated in the work of Roy O. Davies \cite{davies} from the early 60's. He proves that the plane $\mbb R^2$ can be covered by countably many rotated graphs of functions; this was known to be true under the Continuum Hypothesis (proved by Sierpinski in the 30's) while Davies' result holds regardless of cardinal arithmetic. 

The importance of the tree construction is that we can cover arbitrary large structures (in this case $\mbb R^2$) with countable sets in a way that initial segments are fairly close to models (unions of finitely many models). This way the assumption of CH can be eliminated from Sierpinski's original result.

\subsection{The Steinhaus tiling problem} Probably the most important application of Davies-trees is S. Jackson and R. D. Mauldin's solution from 2002 to the Steinhaus tiling problem. In the late 50's H. Steinhaus asked if there is a subset $S$ of $\mbb R^2$ such that every rotation of $S$ tiles the plane or equivalently, $S$ intersects every isometric copy of the lattice $\mbb Z\times \mbb Z$ in exactly one point. Jackson and Mauldin provides an affirmative answer (surveyed in \cite{jackson}); their proof elegantly combines hard combinatorial, geometrical and set theoretical methods (a transfinite induction using Davies-trees).

Again, their proof becomes somewhat simpler assuming CH. However, this assumption can be eliminated, as before, if one uses Davies-trees as a substitute for increasing chains of models.

\subsection{Topology} In 2008, D. Milovich published a paper \cite{milovich} in set theoretic topology (order theory of bases) where he further polished Jackson and Mauldin's Davies-tree decomposition technique. In particular, one can guarantee that the Davies-tree $(M_\alpha)_{\alpha<\kappa}$ has the additional property that $\{N_{\alpha,i}:i<m_\alpha\}\in M_{\alpha}$ for all $\alpha<\kappa$. This extra hypothesis is very useful in several situations.\\

It is likely that there are other papers, even earlier then Davies', where similar techniques appear either explicitly or implicitly however at the point of writing this note we are not aware of further references.

\section{Degrees of disjointness}

We start by proving a simple fact from the theory of almost disjoint set systems.

\begin{dfn} We say that a family of sets $\mc X$ is \emph{$n$-almost disjoint} for some $n\in \mbb N$ iff $|A\cap B|<n$ for every $A\neq B\in \mc X$. $\mc X$ is \emph{essentially disjoint} iff we can select finite $F_A\subs A$ for each $A\in \mc A$ so that $\{A\setm F_A:A\in \mc A\}$ is disjoint.
\end{dfn}

\begin{theorem}[\cite{kopefam}] Every $n$-almost disjoint family $\mc X$ of countable sets is essentially disjoint for every $n\in \mbb N$.
\end{theorem}
\begin{proof} Take a Davies-tree $\mc M= \{M_\alpha:\alpha<\kappa\}$ such that $\mc X\subs \bigcup \mc M$ and that $\mc X\in M_\alpha$ for each $\alpha< \kappa$. Recall that $\bigcup \mc M_{<\alpha}=\bigcup\{N_{\alpha,i}:i<m_{\alpha}\}$ for each $\alpha<\kappa$. We define a map $F$ on $\mc X$ such that $F(A)\in [A]^{<\omega}$ for each $A\in \mc X$ and $\{A\setm F(A):A\in \mc X\}$ is pairwise disjoint. 

Let $\mc X_\alpha=(\mc X \cap M_\alpha)\setm \bigcup \mc M_{<\alpha}$ and $\mc X_{<\alpha}=\mc X \cap (\bigcup \mc M_{<\alpha})$. We define $F$ on each $\mc X_\alpha$ independently so fix $\alpha<\kappa$. 

\begin{obs} $|A\cap (\bigcup \mc X_{<\alpha})|<\omega$ for all $A\in \mc X_\alpha$.
\end{obs}
\begin{proof} Otherwise, there is $i<m_\alpha$ so that $A\cap \bigcup (\mc X\cap N_{\alpha,i})$ is infinite and in particular, we can select $a\in [A\cap \bigcup (\mc X\cap N_{\alpha,i})]^n$. Note that $\bigcup (\mc X\cap N_{\alpha,i})\subs N_{\alpha,i}$ as each set in $\mc X$ is countable hence $a\subs N_{\alpha,i}$ and $a\in N_{\alpha,i}$. However, $N_{\alpha,i}\models$ "there is a \emph{unique} element of $\mc X$ containing $a$" (by $n$-almost disjointness) hence $A\in N_{\alpha,i}\subs \bigcup \mc M_{<\alpha}$ (by elementarity) which contradicts $A\in \mc X_\alpha$.
\end{proof}

Now list $\mc X_\alpha$ as $\{A_{\alpha,l}:l\in \omega\}$. Let $$F(A_{\alpha,l})=A_{\alpha,l}\cap \bigl(\bigcup \mc X_{<\alpha}\cup \bigcup \{A_{\alpha,k}:k<l\}\bigr)$$ for $l<\omega$. Clearly, $F$ witnesses that $\mc X$ is essentially disjoint.

\end{proof}

\section{Clouds above the Continuum Hypothesis}

The next theorem we prove has a certain similarity to Davies' result. The reason that this proof is of greater interest is that it highlights the fact that a set of size $\aleph_n$ can be covered by a Davies-tree such that the initial segments are expressed as the union of $n$ elementary submodels (for $n\in \mbb N$). The same fact is utilized in an application presented in \cite{sigmaD}.  

\begin{dfn}We say that $A\subs \mbb R^2$ is a \emph{cloud around a point $a\in \mbb R^2$} iff every line $l$ through $a$ intersects $A$ in a finite set.
\end{dfn}

Note that one or two clouds cannot cover the plane; indeed, if $A_i$ is a cloud around $a_i$ for $i<2$ then the line $l$ through $a_0$ and $a_1$ intersects $A_0\cup A_1$ in a finite set. How about three or more clouds?

\begin{theorem}[\cite{clouds} and \cite{schmerl}] The following are equivalent for each $n\in \mbb N$:
\begin{enumerate}
	\item $2^\omega\leq \aleph_n$,
	\item $\mbb R^2$ is covered by at most $n+2$ clouds.
\end{enumerate}
\end{theorem}

We only prove (1) implies (2) and follow Komj\'ath's original proof for the $2^\oo=\omg$ case.

\begin{proof} Fix $n\in \omega$ and suppose that the continuum is $\aleph_n$. This implies that $\mbb R^2$ can be covered by a Davies-tree $\{M_\alpha:\alpha<\kappa\}$ so that $\bigcup \mc M_{<\alpha}=\bigcup\{N_{\alpha,i}:i<n\}$ for every $\alpha<\kappa$. % Note that now each initial segment of the sequence of models is expressed as a $n$-term union of elementary submodels.

Fix $n+2$ points $\{a_k:k<n+2\}$ in $\mbb R^2$ in general position (i.e. no three are collinear). Let $\mc L^k$ denote the set of lines through $a_k$ and let $\mc L=\bigcup\{\mc L^k:k<n+2\}$. We will define clouds $A_k$ around $a_k$ by defining a map $F:\mc L \to [\mbb R^2]^{<\oo}$ such that $F(l)\in [l]^{<\oo}$ and letting $$A_k=\{a_k\}\cup \bigcup\{F(l):l\in \mc L^k\}$$ for $k<n+2$. We have to make sure that for every $x\in \mbb R^2$ there is $l\in \mc L$ so that $x\in F(l)$.

Now let $\mc L_\alpha=(\mc L \cap M_\alpha)\setm \bigcup \mc M_{<\alpha}$ and $\mc L_{<\alpha}=\mc L\cap \bigcup \mc M_{<\alpha}$ for $\alpha<\kappa$. We define $F$ on $L_\alpha$ for each $\alpha<\kappa$ independently.

Fix an $\alpha<\kappa$ and list $\mc L_\alpha\setm \mc L'$ as $\{l_{\alpha,j}:j<\oo\}$ where $\mc L'$ is the set of $\binom{n+2}{2}$ lines determined $\{a_k:k<n+2\}$. We let $$F(l_{\alpha,j})=\bigcup\{l\cap l_{\alpha,j}:l\in \mc L'\cup \{l_{\alpha,j'}:j'<j\} \}$$ for $j<\oo$.

We claim that this definition works: fix a point $x\in \mbb R^2$ and we will show that there is $l\in \mc L$ with $x\in F(l)$. Find the unique $\alpha<\kappa$ such that $x\in M_\alpha\setm \bigcup \mc M_{<\alpha}$. It is easy to see that $\cup \mc L'$ is covered by our clouds hence we suppose $x\notin \bigcup \mc L'$. Let $l_k$ denote the line through $x$ and $a_k$. 

\begin{obs} $|\bigcup \mc M_{<\alpha}\cap \{l_k:k<n+2\}|\leq n$.
\end{obs}
\begin{proof} Suppose that this is not true. Then (by the pigeon hole principle) there is $i<n$ such that $|N_{\alpha,i}\cap \{l_k:k<n+2\}|\geq 2$ and in particular the intersection of any two of these lines, the point $x$, is in $N_{\alpha,i}\subs \bigcup \mc M_{<\alpha}$. This contradicts the choice of $\alpha$.
\end{proof}

We have now that $$|\{l_k:k<n+2\}\cap (\mc L_\alpha\setm \mc L')|\geq 2$$ i.e. there is $j'<j<\oo$ such that $l_{\alpha,j'},l_{\alpha,j}\in \{l_k:k<n+2\}$. Hence $x\in F(l_{\alpha,j})$ is covered by one of the clouds.

\end{proof}

\section{The chromatic number and connectivity}

\begin{dfn} The \emph{chromatic number} of a graph $G$ is the least number $\kappa$ such that $G$ is covered by $\kappa$ many independent sets.
\end{dfn}

It is one of the fundamental problems of graph theory how the chromatic number affects the subgraph structure of a graph i.e. is it true that large chromatic number implies the existence of certain obligatory subgraphs? The first result in this area is most likely Mycielski's construction of triangle free graphs of arbitrary large finite chromatic number \cite{myc}. 

It was discovered quite early that a lot can be said about uncountably chromatic graphs; this line of research was initiated by P. Erd\H os and A. Hajnal in \cite{EH0}. One of many problems in that paper asked whether uncountable chromatic number implies the existence of \emph{highly connected} uncountably chromatic subgraphs.

\begin{dfn} A graph $G$ is $n$-connected iff the removal of less than $n$ vertices leaves $G$ connected.
\end{dfn}

Our aim is to prove P. Komj\'ath's following result from \cite{kopeconn}:

\begin{theorem} \label{nconn} Every uncountably chromatic graph $G$ contains $n$-connected uncountably chromatic subgraphs for every $n\in \mbb N$.
\end{theorem}

 Fix a graph $\gr$, $n\in\omega$ and consider the set $\mc A$ of all subsets of $V$ spanning maximal $n$-connected subgraphs of $G$. We let $N_G(v)=\{w\in V: \{v,w\}\in E\}$ for $v\in V$.

We will follow Komj\'ath's framework in the sense that we are going to define a \emph{good ordering} on $\mc A$. The following lemma explains what we mean by good ordering.

\begin{lemma}\label{ordlemma}
 Suppose that $\gr$ is a graph, $\{A_\xi:\xi<\mu\}$ is a cover of $V$ with countably chromatic subsets so that $|N_G(x)\cap \bigcup A_{<\xi}|<\oo$ for all $\xi<\mu$ and $x\in A_\xi\setm \bigcup A_{<\xi}$ where $A_{<\xi}=\{A_\zeta:\zeta<\xi\}$. Then $Chr(G)\leq \oo$.

\end{lemma}

\begin{proof} Suppose that $g_\xi:A_\xi\to \oo$ witnesses that the chromatic number of $A_\xi$ is $\leq \oo$. We define $f:V\to \oo\times \oo$ by defining $f\uhp (A_\xi\setm \bigcup A_{<\xi} )$ by induction on $\xi<\mu$. If $x\in A_\xi\setm \bigcup A_{<\xi}$ then the first coordinate of $f(x)$ is $g_\xi(x)$ while the second coordinate of $f(x)$ avoids all the finitely many second coordinates appearing in $\{f(y):y\in N_G(x)\cap \bigcup A_{<\xi}\}$. It is easy to see that $f$ witnesses that $G$ has countable chromatic number.
\end{proof}

%\begin{proposition}
 %Suppose that $Chr(G)>\oo$ and $n\in\omega$. Then there is a subgraph of $G$ which is $n$-connected and has minimal degree $\omega$. 
%\end{proposition}

Let us continue with some straightforward observations about the maximal $n$-connected sets:

\begin{obs}\label{inobs}
\begin{enumerate}
\item $A\not\subseteq A'$ for all $A\neq A'\in \mc A$,
\item $|A\cap A'|<n$ for all $A\neq A'\in \mc A$,
\item $|\{A\in \mc A: a\subs A\}|\leq 1$ for all $a\in [V]^{\geq n}$,
\item $|N_G(x)\cap A|<n$ for all $x\in V\setm A$ and $A\in \mc A$.
\end{enumerate}
\end{obs}

The next claim is fairly simple and describes a situation when we can join $n$-connected sets.

\begin{claim}\label{amalg1} Suppose that $A_i\subs V$ spans an $n$-connected subset for each $i<n$ and we can find $Y=\{y_{i,k}:i<n, k<n\}$ and $X=\{x_k:k<n\}$ distinct points so that $$y_{i,k}\in A_i\cap N_G(x_k)$$ for all $i<n,k<n$. Then $A=\bigcup \{A_i:i<n\}\cup X$ is $n$-connected.
\end{claim}
\begin{proof}
 Let $F\in [A]^{<r}$ and note that there is a $k<n$ so that $\{y_{i,k},x_k:i<n\}\cap F=\emptyset$ for some $k<n$. Thus $\cup \{A_i:i<n\}\cup\{y_{i,k},x_k:i<n\}\setm F$ is connected as $A_i\setm F$ is connected for all $i<n$. Finally, if $x_j\in A\setm F$ then $N_G(x_j)\cap \cup\{A_i:i<n\}\setm F\neq \emptyset$ so we are done.
\end{proof}

\begin{comment}
\begin{claim}\label{amalg2} Suppose the sets $A_i,B_i\subs V$ are $n$-connected for $i< n$ and there is a sequence of distinct points $\{x_{i,j}:i<n,j<n\}$ so that $$x_{i,j}\in A_i\cap B_j$$ for all $i,j<n$. Then $\cup\{A_i,B_i:i<n\}$ is $n$-connected.
\end{claim}
\begin{proof}
 The proof is very similar to the preceding argument.
\end{proof}

\end{comment}

Now, we deduce some useful facts about elementary submodels and maximal $n$-connected sets.

\begin{lemma} Suppose that $N\prec H(\Theta)$ with $G\in N$ and $$|N_G(x)\cap N|\geq n$$ for some $x\in V\setm N$. Then $x\in A$ for some $A\in\mc A\cap N$.
\end{lemma}
\begin{proof}
 Let $a\in [N_G(x)\cap N]^{n}$. There is a copy of $K_{n,\omg}$ (complete bipartite graph with classes of size $n$ and $\omg$) which contains $a\cup \{x\}$; to see this, use Fact \ref{elfact} to $X=\bigcap\{N_G(y):y\in A\}$. As $K_{n,\omg}$ is $n$-connected, there must be $A\in \mc A$ with $a\cup\{x\}\subs A$ as well. Also, there is $A'\in \mc A\cap N$ with $a\subs A'$ by elementarity; as $|A\cap A'|\geq n$ we have $A=A'$ which finishes the proof.
\end{proof}

\begin{lemma}\label{mlemma1} Suppose that $N\prec H(\Theta)$ with $G\in N$ and $$|N_G(x)\cap \bigcup(\mc A\cap N)|\geq\omega$$ for some $x\in V\setm N$. Then $x\in A$ for some $A\in\mc A\cap N $.
\end{lemma}
\begin{proof}
 Suppose that the conclusion fails; by the previous lemma, we have $|N_G(x)\cap N|< n$. 
In particular, there is sequence of distinct $A_i\in \mc A\cap N$ for $i<n$ so $$(N_G(x)\cap A_i)\setm N\neq \emptyset$$ for all $i<n$ (as $N_G(x)\cap A$ is finite if $A\in N\cap \mc A$). 

Thus $$N\models \forall F\in[V]^{<\omega}\exists x\in V\setm F\text{ and } y_i\in (A_i\cap N_G(x))\setm F.$$

Now, we can find distinct $\{y_{i,k}:i<n, k<n\}$ and $X= \{x_k:k<n\}$ so that $$y_{i,k}\in A_i\cap N_G(x_k).$$ Finally, $\cup \{A_i:i<n\}\cup X$ is $n$-connected by Claim \ref{amalg1} which contradicts the maximality of $A_i$. 
\end{proof}

\begin{comment}

\begin{lemma}\label{mlemma2}Suppose that $A\in \mc A$ and $$|A\cap \bigcup (\mc A\cap N)|\geq \omega.$$ Then $A\in \mc A \cap N$.
\end{lemma}
\begin{proof}
 First of all, observe that $|A\cap N|\geq n$ implies $A\in \mc A\cap N$ for all $A\in \mc A$; indeed, for any $a\in [A\cap N]^n$ there is exactly one $A\in \mc A$ which contains $a$. Thus $a\in N$ implies $A\in N$.

Now, suppose that there is $A\in \mc A\setm N$ with $|A\cap \bigcup (\mc A\cap N)|\geq \omega$; by the previous observation, we can suppose that $(A\cap \bigcup (\mc A\cap N))\setm N$ is infinite. As $\mc A$ is almost disjoint, we can find $A_i\in \mc A\cap N$ and distinct $x_i\in A_i\cap A$ for $i<n$ (note that $\{A'\setm N:A'\in N\cap \mc A\}$ is a disjoint family).

Thus $$N\models \forall F\in[V\cup \mc A]^{<\omega}\exists x_i\in V\setm F\text{ and } B\in \mc A\setm F \text { with } x_i\in B\cap A_i \text{ for }i<n.$$

Thus we can find $B_i\in \mc A$ for $i< n$ and points $\{x_{i,j}:i<n,j<n\}$ so that $$x_{i,j}\in A_i\cap B_j$$ for all $i,j<n$. However, this implies that $\cup\{A_i,B_i:i<n\}$ is $n$-connected by Claim \ref{amalg2}; this contradicts the maximality of $A_i\in \mc A$.
\end{proof}

\end{comment}

\begin{proof}[Proof of Theorem \ref{nconn}]
Let $G,\mc A$ be as above and suppose that every $A\in \mc A$ is countably chromatic; we will show that in this case, $G$ is countably chromatic.

First, we prove that $\bigcup \mc A$ is countably chromatic. Take a Davies-tree covering $\mc A$ i.e. a sequence $(M_\alpha)_{\alpha<\kappa}$ of countable elementary submodels such that for all $\alpha<\kappa$ there is a  \emph{finite} sequence of elementary submodels $(N_{\alpha,j})_{j<m_\alpha}$ so that $$\bigcup \mc M_{<\alpha}=\bigcup\{N_{\alpha,j}:j<m_\alpha\},$$
with $G\in M_\alpha\cap N_{\alpha,j}$ and $\mc A\subs \bigcup\{M_\alpha:\alpha<\kappa\}$.

Let $\mc A_{<\alpha}=\mc A\cap \bigcup \mc M_{<\alpha}$ and $\mc A_\alpha=(\mc A\cap M_\alpha)\setm \mc A_{<\alpha}$ for $\alpha<\kappa$. Well order $\mc A$ as $\{A_\xi:\xi<\mu\}$ so that 
\begin{enumerate}
\item $A_\zeta\in \mc A_{<\alpha},A_\xi\in \mc A\setm \mc A_{<\alpha}$ implies $\zeta<\xi$ and
\item $\mc A_\alpha\setm \mc A_{<\alpha}$ has order type $\leq \omega$
\end{enumerate}

for all $\alpha<\kappa$. 

We claim that the above enumeration of $\mc A$ satisfies Lemma \ref{ordlemma} and thus $\bigcup \mc A$ is countably chromatic. By the second property of our enumeration and Observation \ref{inobs} (4), it suffices to show that $$|N_G(x)\cap\bigcup \mc A_{<\alpha}|<\oo $$ if $x\in A\setm \bigcup \mc A_{<\alpha}$ for all $A\in \mc A_\alpha\setm \mc A_{<\alpha}$ and $\alpha<\kappa$.

However, as $\mc A_{<\alpha}=\bigcup\{\mc A\cap N_{\alpha,j}:j<m_\alpha\}$, this should be clear from applying Lemma \ref{mlemma1} for each of the finitely many models $N_{\alpha,j}$ where $j<m_\alpha$.

Now, we show that $G$ is countably chromatic; otherwise, the graph spanned by $V\setminus \bigcup \mc A$ is uncountably chromatic. However, every uncountably chromatic graph, and so $V\setminus \bigcup \mc A$ as well, contains an $n$-connected subgraph (actually a copy of $K_{n,\omg}$ by \cite{EH0}) which contradicts the definition of $\mc A$.
\end{proof}

We note that Komj\'ath also proves that every uncountably chromatic subgraph contains an $n$-connected uncountably chromatic subgraph with minimal degree $\omega$; we were not able to deduce this stronger result with our tools.

It is an open problem whether every uncountably chromatic graph $G$ contains an $\omega$-connected subgraph \cite{koperev} (i.e. removing finitely many vertices leaves the graph connected). It was recently proved however that

\begin{theorem}[\cite{trees}] There is a graph of chromatic number $\omg$ and size $2^\oo$ such that every uncountable set is separated by a finite set. In particular, every $\omega$-connected subset is countable.
\end{theorem}

\section{Future work}

There are great possibilities in the use of Davies-trees beyond finding new proofs or eliminating CH from known results (which already is a great deal). Recently, L. Soukup started to develop the analogue of Davies-trees with $\sigma$-closed models.

\begin{theorem}[\cite{sigmaD}] Suppose $V=L$. Then for every cardinal $\kappa$ there is a sequence $(M_\alpha)_{\alpha<\kappa}$ of elementary submodels of $H(\Theta)$ covering $\kappa$ such that 
\begin{enumerate}
	\item $[M_\beta]^\oo\subs M_\beta$ and $|M_\beta|=\omg$,
	\item there are $N_{\beta,j}\prec H(\Theta)$ with $[N_{\beta,j}]^\oo\subs N_{\beta,j}$ for $j<\oo$ such that $$\bigcup\{M_{\alpha}:\alpha<\beta\}=\bigcup\{N_{\beta,j}:j<\oo\}$$
\end{enumerate} for all $\beta<\kappa$.
\end{theorem}

See the presentation \cite{sigmaD} for more on $\sigma$-Davies-trees and further applications of ordinary Davies-trees.

\end{document}